\documentclass[12pt,a4paper]{article}
\usepackage[top=1in, bottom=1in, left=1in, right=1in]{geometry}
\usepackage{amsmath}
\usepackage{amssymb}
\usepackage{times}
\usepackage{graphicx,epstopdf}
\usepackage{ mathrsfs }
\usepackage{graphics}
\usepackage{hyperref}
\usepackage{amsthm}
\usepackage[skip=2pt]{caption}
\usepackage[figurename=Fig.]{caption}
\usepackage[usenames,dvipsnames]{pstricks}
\usepackage{epsfig}

\allowdisplaybreaks
\newtheorem{thm}{Theorem}[section]
\theoremstyle{definition}

\newtheorem{cor}{Corollary}[section]
\newtheorem{lem}{Lemma}[section]
\begin{document}
\title{\textbf{Lih Wang and Dittert's Conjectures on Permanents}}
\author{Divya.K.U and K. Somasundaram}
\date{\small{Department of Mathematics, Amrita School of Physical Sciences, Coimbatore \\ Amrita Vishwa Vidyapeetham, India.\\ 
 divyaku93@gmail.com, s\_sundaram@cb.amrita.edu}}

\maketitle
\begin{abstract}
\indent Let $\Omega_n$ denote the set of all doubly stochastic matrices of order $n$. Lih and Wang conjectured that for $n\geq3$, per$(tJ_n+(1-t)A)\leq t $per$J_n+(1-t)$per$A$, for all $A\in\Omega_n$ and all $t \in [0.5,1]$, where $J_n$ is the $n \times n$ matrix with each entry equal to $\frac{1}{n}$. This conjecture was proved partially for $n \leq 5$. \\
\indent Let $K_n$ denote the set of non-negative $n\times n$ matrices whose elements have sum $n$. Let $\phi$ be a real valued function defined on $K_n$ by $\phi(X)=\prod_{i=1}^{n}r_i+\prod_{j=1}^{n}c_j$ - per$X$ for $X\in K_n$ with row sum vector $(r_1,r_2,...r_n)$ and column sum vector $(c_1,c_2,...c_n)$. A matrix $A\in K_n$ is called a $\phi$-maximizing matrix if $\phi(A)\geq \phi(X)$ for all $X\in K_n$. Dittert conjectured that $J_n$ is the unique $\phi$-maximizing matrix on $K_n$. Sinkhorn proved the conjecture for $n=2$ and Hwang proved it for $n=3$. \\
\indent In this paper, we prove the Lih and Wang conjecture for $n=6$ and Dittert conjecture for $n=4$.     
\end{abstract}

%--------------------------------------------
\maketitle
\section{Introduction}
\indent The permanent of an $n\times n$ matrix $A=(a_{ij})$  is defined by
\begin{equation*}
\mbox{per}(A)=\sum\limits_{\sigma\in S_n} a_{1\sigma(1)}a_{2\sigma(2)}...a_{n\sigma(n)},
\end{equation*}
where $S_n$ is the symmetric group of degree $n$.\\
\indent Let $\Omega_n$  denote the set of all  doubly stochastic matrices of order $n$ and let $J_n$ be an $n\times n$ matrix with each entry equal to $\frac{1}{n}$.\\
\indent For positive integers $n$ and $k$ with $\left(1\leq k\leq n\right)$, $Q_{k,n}$ denotes the set\\ $\{\left( i_1,..., i_k\right)/ 1\leq i_1 < ...< i_k\leq n \}$.
  For $\alpha, \beta \in Q_{k,n}$, let $A \left(\alpha / \beta \right)$ be the submatrix of  $A$ obtained by deleting the rows indexed by $\alpha$ and columns indexed by $\beta$ and $A[\alpha / \beta]$ be the submatrix of $A$ with rows and columns indexed by $\alpha$ and $\beta$ respectively.\\
\indent For $1\leq k \leq n$, the $k^{th}$ order subpermanent of $A$ is defined by
 \begin{equation*}
\sigma_k(A)=\sum\limits_{\alpha,\beta \in Q_{k,n}}\text{per}\left(A[\alpha / \beta]\right)
\end{equation*}
In this paper, we use the following results quoted by Minc \cite{HMI1978}:\\
 If $A$ and $B$ are two $n\times n$ matrices and $1\leq k\leq n$, then
  \begin{equation*}
  \text{per}(A)=\sum\limits_{\beta \in Q_{k,n}} \text{per}\left(A[\alpha / \beta]\right)\text{per}\left(A\left(\alpha / \beta\right)\right),  \text{for }  \alpha\in Q_{k,n},
  \end{equation*}
   and
  \begin{equation*}
  \text{per}(A+B)=\sum\limits_{k=0}^n  S_k \left(A, B\right),
      \end{equation*}
       where $S_k\left(A, B\right)=\sum\limits_{\alpha,\beta \in Q_{k,n}}\text{per}(A[\alpha / \beta])\text{per}(B\left(\alpha / \beta\right))$,
    $ $per$ \left(A[\alpha / \beta]\right)=1$ when $k=0$ \\ and $ $per$ \left(B\left(\alpha / \beta\right)\right)=1$ when $k=n$.\\
\indent Several authors have considered the problem of finding an upper bound for the permanent of a convex combination of $J_n$ and $A$, where $A\in\Omega_n$. In particular, Marcus and Minc \cite{MMA1967} have proposed a conjecture, which states that if $A\in\Omega_n, n\geq 2$, then per$\left(\frac{nJ_n-A}{n-1}\right)\leq$ per$(A),$ with equality holding if and only if $A=J_n, n\geq 3$.  They established that the conjecture is true for $n=2$, or if $ A$ is a positive semi-definite symmetric matrix, or if $A$ is in a sufficiently small neighbourhood of $J_n$.
    Wang \cite{ETH1977} proved the Marcus - Minc conjecture for $n=3$ with a revised statement of the case of equality and he proposed a conjecture that per$\left(\frac{nJ_n+A}{n+1}\right)\leq$per$(A)$.\\
     \indent Rather than consider a particular convex combination of $A$ with $J_n$, Foregger \cite{THF1979} raised a question  whether
\begin{equation}
  \mbox{per} \left(tJ_n+\left(1-t\right) A\right)\leq \mbox{per}(A)
 \end {equation}
  holds for all $A\in \Omega_n$ and $0<t \leq\frac{n}{n-1}$? He proved that the inequality (1) holds for $0<t \leq\frac{3}{2}$ and $A\in \Omega_3$ with equality if and only if $A=J_3$ or $t=\frac{3}{2}$ and $A=\frac{1}{2}\left(I+P\right)$ \\ (up to permutations of rows and columns), where $P$ is a full cycle permutation matrix. Later, Foregger himself proved in \cite{THF1988} that the inequality (1) holds for $n=4$ and  $t_0<t \leq\frac{4}{3},$  where $t_0$ is approximately equal to 0.28095. \\
 \indent Another conjecture in permanents is due to Holens \cite{FHO1964} and Dokovic \cite{DZD1967}, namely the Holens - Dokovic conjecture, which states that if $A\in\Omega_n $ and $k$ is an integer, $1\leq k\leq n$, then
 \begin{equation}
 \sigma_k(A)\geq\frac{(n-k+1)^2}{kn}\sigma_{k-1}(A).
 \end{equation}
  They proved this conjecture for $k\leq3$.  Kopotun \cite{KAK1994} proved the inequality (2) for $k=4$ and $n\geq 5$. \\
  \indent By Minc \cite{HMI1978}, the inequality $(2)$ is equivalent to the Monotonicity conjecture, which states that if $A$ is any matrix in $\Omega_n$ and $2\leq k \leq n$, then $\sigma_k(X)$ is monotonically increasing on the line segment from $J_n$ to $A$. For $k=n$,  the monotonicity conjecture on permanents is equivalent to 
\begin{equation*}
	\mbox{per}((1-t)J_n+t A)\leq \mbox{per}(A)
\end{equation*} 
for all $A\in\Omega_n$ and $0\leq t \leq 1$ \cite{SGH1991, KAK1996}. Most interestingly, Wanless \cite{IMW1999} disproved the Holens - Dokovic inequality (2).\\

\indent Brualdi and Newman \cite{RAB1965} considered the inequality
 \begin{equation}
   \text{per} \left(\alpha J_n + \left(1-\alpha\right) A \right) \leq \alpha \text{ per}(J_n) + \left(1-\alpha\right)\text{per}(A)
 \end{equation}
  and showed that it did not always hold for all $A\in \Omega_n$ and  for all $\alpha \in [0, 1]$.\\
For $n=3$, Lih and Wang \cite{ETH1982} proved that (3) is true for all $A \in \Omega_n$ and  for all $\alpha \in [\frac{1}{2}, 1]$ and they proposed the following conjecture (quoted by Cheon and Wanless \cite{GSC2005}): The inequality  (3) is true for all $A \in \Omega_n$  and for all $\alpha \in [\frac{1}{2},1]$. For $n=4$, Foregger \cite{THF1988} proved that the inequality $(3)$ holds for all $A\in\Omega_4$ and for all $\alpha \in [0.63, 1]$. Kopotun \cite{KAK1996} extended this inequality $(3)$ to subpermanents. Subramanian and Somasundaram \cite{SSK2016} proved that the inequality $(3)$ holds for $n=4$ and $n=5$ with some conditions on $A$. A detailed survey of conjectures on permanent was given by Gi-Sang Cheon and Wanless \cite{GSC2005} and Fuzhen Zhang \cite{FUAH2016}.  In the next section, we prove that the Lih Wang conjecture is true for $n=4$ and the inequality (3) is valid in $[\beta,1],$ where the value of $\beta$ is about 0.486. Also, we prove the Lih Wang Conjecture for $n=6$ with some conditions on $A$ in the interval [0.7836, 1]. 
The following theorem is due to Foregger \cite{THF1988}.\\
\begin{thm}
Let $A\in\Omega_4$ and $t_2<t\leq1,$ where $t_2$ is the unique real root of the polynomial $106t^3-418t^2+465t-153.$ Then\begin{center}$per(tJ_4+(1-t)A)\leq tper(J_4)+(1-t)perA$\end{center} with equality iff $A=J_4$.
\end{thm}
He found that the value of $t_2$ is about $0.6216986477375.$\vskip0.5cm The following theorem is due to Subramanian and Somasundaram \cite{SSK2016}
\begin{thm}
	Let $A\in\Omega_n$ and $2\leq k\leq n$. If the polynomial $\sum_{r=2}^{k}r \frac{(k-r)!}{n^{k-r}}{{n-r}\choose{k-r}}^2\sigma_r(A-J_n)t^{r-2}$ has no root in $(0,1)$ then $\sigma_k(A)\geq\frac{(n-k+1)^2}{nk}\sigma_{k-1}(A),k=2,3,...,n.$
	\end{thm}
Another conjecture on permanents is Dittert's conjecture [\cite{HMI1978}, conjecture 28].
 Let $K_n$ denote the set of all $n \times n$ real nonnegative matrices whose entries have sum $n$.  Let $\phi$ denote a real valued function defined on $K_n$ by $\phi(X)=\prod_{i=1}^{n}r_i+\prod_{j=1}^{n}c_j-per(X)$ for $X \in K_n$ with row sum vector $R_{X}=(r_1,r_2,...r_n)$ and column sum vector $C_{X}=(c_1,c_2,...c_n)$. For the same $X$, let $\phi_{ij}(X)=\prod_{k \neq i}r_k+\prod_{l \neq j}c_l-per X(i|j)$.\vskip0.5cm A matrix $A \in K_n$ is called a $\phi$-maximizing matrix on $K_n$ if $\phi(A)\geq\phi(X)$ for all $X\in K_n$. An $n \times n$ is called fully indecomposable if it does not contain an $s \times t$ zero submatrix with $s + t = n.$\vskip0.5cmE. Dittert conjectured that [\cite{HMI1978}, conjecture 28] $J_n$ is the unique $\phi$-maximizing matrix on $K_n$. R.Sinkhorn\cite{RS1984} proved that every $\phi$-maximizing matrix on $K_n$ has a positive permanent and also that the conjecture is true for $n=2$.\\ \indent Suk Geun Hwang\cite{SGH1986, SGH1987} investigated some properties of $\phi$-maximizing matrices. The following Lemmas and theorems are due to SG Hwang \cite{SGH1987}.\\
\begin{lem}
Let $A$ be a $\phi$-maximizing matrix on $K_n$, and let $1\leq s < t \leq n.$ If columns $s$ and $t$ of $A$ have either the same sums or the same (0, 1) patterns, then the matrix obtained from $A$ by replacing each of columns $s$ and $t$ by its average is also a $\phi$-maximizing matrix on $K_n$. A similar statement holds for rows.
\end{lem} 
He also proved that Dittert's conjecture is true for positive semidefinite symmetric matrices in $K_n$ and for matrices in a sufficiently small neighbourhood of $J_n$ in $K_n$. 
\begin{lem}
 Let $A=[a_{ij}]$ be a  $\phi$-maximizing matrix on $K_n$. Then
\begin{itemize}
 \item $\phi_{ij}(A) = \phi_{kl}(A)$ if $a_{ij}>0$ and $a_{kl}>0$.
 \item $\phi_{ij}(A) \leq\phi_{kl}(A)$ if $a_{ij}=0$ and $a_{kl}>0$.
 \end{itemize}
\end{lem}

\begin{thm}
If $A =[a_{ij}]$ is a $\phi$-maximizing matrix on $K_n$,
 	then $\phi_{ij}(A) = \phi(A)$ if $a_{ij}>0$ and $\phi_{ij}(A) \leq\phi(A)$ if $a_{ij}=0$.
	\end{thm}
\begin{thm} $J_3$ is the unique $\phi$-maximizing matrix on $K_3$.
\end{thm}
He  \cite{SGH1987} stated that one can at lest prove that if $A$ is a $\phi$ maximizing matrix on $K_n$ then $A$ is fully indecomposable. Based on his remark we prove that if $A$ is a $\phi$-maximizing matrix on $K_4$ then $A$ is fully indecomposable.\\
Cheon and Wanless \cite{CW2012} proved that if $A \in K_n$ is partly decomposable then $\phi(A)<\phi(J_n)$ and that if the zeroes in $A \in K_n$ form a block then $A$ is not a $\phi$-maximizing matrix. Cheon and Yoon \cite{CY2006} obtained some sufficient conditions for which the Dittert conjecture holds.

\section{Lih Wang conjecture}
In this section, first we prove the Lih Wang conjecture completely for $n=4$. We improved the results of Theorem 1.1 and Theorem 1.2. Later in Theorem 2.2, we prove this conjecture partially for $n=6$.
\begin{thm}
If $A\in\Omega_4$ then $per(\alpha J_4+(1-\alpha)A)\leq\alpha per(J_4)+(1-\alpha)per A$  for all $\alpha \in [t,1]$  where the value of $t$ is about $0.485.$
\end{thm}
\begin{proof}
Let $f(\alpha) = \alpha$ per$(J_4)+(1-\alpha)$per$A$ - per$(\alpha J_4+(1-\alpha)A)$.\\
As we know that per$(J_4)=\frac{3}{32}$, \\
\begin{equation} 
f(\alpha)=3\alpha+32(1-\alpha) \text{per} (A)-32 \text{per}(\alpha J_4+(1-\alpha)A)\geq 0.
\end{equation} \\
Let $a_i$ denote the $i^{th}$ column of $A$ and  $e_r(\textbf{x})=e_r(x_1,x_2,...x_n)$ denote the $r^{th}$ elementary symmetric function of $\textbf{x}=(x_1,x_2...x_n)$. That is, the sum of products of components taken $r$ at a time, $r=1,2,...n$. Let $T_r(A)$ denote the set of ${n}\choose{r}$  sums of columns of $A$ taken $r$ at a time. That is, let $T_r(A)=\{\textbf{x}=a_{i_1}+a_{i_2}+...+a_{i_r}|(i_1,i_2,...,i_r)$ is a $r$-subset of $(1,2,...n)\},$\\ then \\$$\text{per}A=\sum_{T_n(A)}e_n(\textbf{x})-\sum_{T_{n-1}(A)}e_n(\textbf{x})+...+(-1)^{n-1}\sum_{T_1(A)}e_n(\textbf{x}).$$\\
When $A$ is doubly stochastic we have $$\sum_{T_n(A)}e_n(\textbf{x})=1.$$

Let $B=\alpha J_4+(1-\alpha)A$. We use the formula of Eberlein and Mudholkar (\cite{EGM1968}, page 392) to calculate permanent of $B$.
\begin{equation*} \text{per}(B)=\frac{-1}{3}+\frac{1}{9}\sum_{T_1( B)}(-4e_2+9e_3-18e_4)(x)+\frac{1}{18}\sum_{T_2(
B)}(4e_2-9e_3+18e_4)(\textbf{x}).
\end{equation*}
The elements of $T_1(B)$ are the columns of $B$. Let $b_i$ denote the $i^{th}$ column of $B$. Then $b_i$ is of the form $b_i=\alpha\left[ 
{\begin{array}{c}\frac{1}{4}\\\frac{1}{4}\\\frac{1}{4}\\ \frac{1}{4} \end{array} }\right]+(1-\alpha)\left[ {\begin{array}{c}a_{1i}\\a_{2i}\\a_{3i}\\a_{4i}\\\end{array} } \right]$,  where $\left[{\begin{array}{c}a_{1i}\\a_{2i}\\a_{3i}\\a_{4i}\\\end{array}}\right]$ is the $i^{th}$ column of $A$.\\

\noindent Now, $e_2(b_i)=e_2 \left[\begin{array}{c}\frac{\alpha}{4}+(1-\alpha)a_{1i}\\ \frac{\alpha}{4}+(1-\alpha)a_{2i}\\ \frac{\alpha}{4}+(1-\alpha)a_{3i}\\ \frac{\alpha}{4}+(1-\alpha)a_{4i}\\ \end{array}  \right] \\\\
=[\frac{\alpha}{4}+(1-\alpha)a_{1i}][\frac{\alpha}{4}+(1-\alpha)a_{2i}]+[\frac{\alpha}{4}+(1-\alpha)a_{1i}][\frac{\alpha}{4}+(1-\alpha)a_{3i}]+...\\\\=\frac{3}{8}\alpha^2+\frac{3}{4}\alpha(1-\alpha)+(1-\alpha)^2e_2(a_i).$\\

\noindent $e_3(b_i)=[\frac{\alpha}{4}+(1-\alpha)a_{1i}][\frac{\alpha}{4}+(1-\alpha)a_{2i}][\frac{\alpha}{4}+(1-\alpha)a_{3i}]+[\frac{\alpha}{4}+(1-\alpha)a_{1i}][\frac{\alpha}{4}+(1-\alpha)a_{2i}] [\frac{\alpha}{4}+(1-\alpha)a_{4i}]+...\\\\
=\frac{\alpha^3}{16}+3\frac{\alpha^2(1-\alpha)}{16}+\frac{\alpha(1-\alpha)^2e_2(a_i)}{2}+(1-\alpha)^3e_3(a_i).$\\

\noindent Similarly, $e_4(b_i)=[\frac{\alpha}{4}+(1-\alpha)a_{1i}][\frac{\alpha}{4}+(1-\alpha)a_{2i}][\frac{\alpha}{4}+(1-\alpha)a_{3i}][\frac{\alpha}{4}+(1-\alpha)a_{4i}]\\\\
=\frac{\alpha^4}{256}+\frac{\alpha^3(1-\alpha)}{64}+\frac{\alpha^2(1-\alpha)^2e_2(a_i)}{16}+\frac{\alpha(1-\alpha)^3e_3(a_i)}{4}+(1-\alpha)^4e_4(a_i).$\\\\If $\textbf{x}\in T_2(B)$, then $\textbf{x}$ is of the form $\textbf{x}=\left[ {\begin{array}{c}\frac{\alpha}{2}+(1-\alpha)(a_{1i}+a_{1j})\\\frac{\alpha}{2}+(1-\alpha)(a_{2i}+a_{2j})\\\frac{\alpha}{2}+(1-\alpha)(a_{3i}+a{3j})\\\frac{\alpha}{2}+(1-\alpha)(a_{4i}+a_{4j})\\\end{array} } \right].$\\

\noindent Then $e_2(\textbf{x})=[\frac{\alpha}{2}+(1-\alpha)(a_{1i}+a_{1j})][\frac{\alpha}{2}+(1-\alpha)(a_{2i}+a_{2j})]+[\frac{\alpha}{2}+(1-\alpha)(a_{1i}+a_{1j})][\frac{\alpha}{2}+(1-\alpha)(a_{3i}+a_{3j})]+...\\\\=\frac{3}{2}{\alpha^2}+3\alpha(1-\alpha)+(1-\alpha)^2e_2(\textbf{y})$, where $\textbf{y}\in T_2(A)$.
\\\\
$e_3(\textbf{x})=[\frac{\alpha}{2}+(1-\alpha)(a_{1i}+a_{1j})][\frac{\alpha}{2}+(1-\alpha)(a_{2i}+a_{2j})][\frac{\alpha}{2}+(1-\alpha)(a_{3i}+a_{3j})]+[\frac{\alpha}{2}+(1-\alpha)(a_{1i}+a_{1j})][\frac{\alpha}{2}+(1-\alpha)(a_{2i}+a_{2j})][\frac{\alpha}{2}+(1-\alpha)(a_{4i}+a_{4j})]+...$\\\\=$\frac{\alpha^3}{2}+\frac{3}{2}\alpha^2(1-\alpha)+\alpha(1-\alpha)^2e_2(\textbf{y})+(1-\alpha)^3e_3(\textbf{y})$, where $\textbf{y}\in T_2(A)$.\\\\
$e_4(\textbf{x})=[\frac{\alpha}{2}+(1-\alpha)(a_{1i}+a_{1j})][\frac{\alpha}{2}+(1-\alpha)(a_{2i}+a_{2j})][\frac{\alpha}{2}+(1-\alpha)(a_{3i}+a_{3j})][\frac{\alpha}{2}+(1-\alpha)(a_{4i}+a_{4j})]$\\\\$=\frac{\alpha^4}{16}+\frac{(\alpha^3)(1-\alpha)}{4}+\frac{1}{4}\alpha^2(1-\alpha)^2e_2(\textbf{y})+\frac{1}{2}\alpha(1-\alpha)^3e_3(\textbf{y})+(1-\alpha)^4e_4(\textbf{y})$, where $\textbf{y}\in T_2(A)$.\\\\

\noindent Therefore,\\
$$\sum_{i=1}^{4}e_2(b_i)=\frac{3}{2}\alpha^2+3\alpha(1-\alpha)+(1-\alpha)^2\sum_{i=1}^{4}e_2(a_i),$$\\
$$\sum_{i=1}^{4}e_3(b_i)=\frac{\alpha^3}{4}+\frac{3}{4}\alpha^2(1-\alpha)+\frac{1}{2}\alpha(1-\alpha)^2\sum_{i=1}^{4}e_2(a_i)+(1-\alpha)^3\sum_{i=1}^{4}e_3(a_i) $$\\ and  
$$\sum_{i=1}^{4}e_4(b_i)=\frac{1}{64}\alpha^4+\frac{1}{16}\alpha^3(1-\alpha)+\frac{1}{16}\alpha^2(1-\alpha)^2\sum_{i=1}^{4}e_2(a_i)+\frac{1}{4}\alpha(1-\alpha)^3\sum_{i=1}^{4}e_3(a_i)+(1-\alpha)^4\sum_{i=1}^{4}e_4(a_i).$$\\
$$\sum_{T_2(B)}e_{2}(\textbf{x})=18\alpha-9\alpha^2+(1-\alpha)^2\sum_{T_2(A)}e_2(\textbf{y}),$$\\
$$\sum_{T_2(B)}e_{3}(\textbf{x})=-6\alpha^3+9\alpha^2+\alpha(1-\alpha)^2\sum_{T_2(A)}e_2(\textbf{y})+(1-\alpha)^3\sum_{T_2(A)}e_3(\textbf{y}) $$\\ and
$$\sum_{T_2(B)}e_{4}(\textbf{x})=\frac{3\alpha^3}{2}-\frac{9\alpha^4}{8}+\alpha(1-\alpha)^2\sum_{T_2(A)}e_2(\textbf{y})+(1-\alpha)^3\sum_{T_2(A)}e_3(\textbf{y}).$$\\\\

\noindent Now, we use the formula of Eberlein and Mudholkar (\cite{EGM1968}, equations (3.4) and (3.5)) for any $A\in\Omega_4$, $$\sum_{T_2(A)}e_2(\textbf{x})=2\sum_{T_1(A)}e_2(\textbf{x})+6,$$\\
$$\sum_{T_2(A)}e_3(\textbf{x})=2\sum_{T_1(A)}e_2(\textbf{x}).$$
\vskip0.2cm \noindent Using the above two equations, we get $$\sum_{T_2(B)}e_{2}(\textbf{x})=18\alpha-9\alpha^2+(2-4\alpha+2\alpha^2)\sum_{i=1}^{4}e_{2}(a_i)+6-12\alpha+6\alpha^2,$$\\
$$\sum_{T_2(B)}e_{3}(\textbf{x})=-6\alpha^3+9\alpha^2+(\alpha-2\alpha^2+\alpha^3)(2\sum_{i=1}^{4}e_{2}(a_i)+6)+(2-6\alpha+6\alpha^2-2\alpha^3)\sum_{i=1}^{4}e_{2}(a_i)$$ \\ and $$ \sum_{T_2(B)}e_{4}(\textbf{x})=\frac{3\alpha^3}{2}-\frac{9\alpha^4}{8}+(\frac{\alpha^2-2\alpha^3+\alpha^4}{4})(2\sum_{i=1}^{4}e_2(a_i)+6) + $$ \\ $$ (\alpha-3\alpha^2+3\alpha^3-\alpha^4)\sum_{i=1}^{4}e_2(a_i)+(1-\alpha)^4\sum_{T_2(A)}e_4(\textbf{y}).$$
\\
Therefore $$\text{per}(B)=\frac{-1}{3}+\frac{1}{9}(-4(\frac{3\alpha^2}{2}+3\alpha(1-\alpha)+(1-\alpha)^2\sum_{i=1}^{4}e_2(a_i))$$\\
$$+9(\frac{\alpha^3}{4}+\frac{3\alpha^2(1-\alpha)}{4}+\frac{\alpha(1-\alpha)^2}{2}\sum_{i=1}^{4}e_2(a_i)+(1-\alpha)^3\sum_{i=1}^{4}e_3(a_i))$$\\$$-18(\frac{\alpha^4}{64}+\frac{\alpha^3(1-\alpha)}{16}+\frac{\alpha^2(1-\alpha)^2}{16}\sum_{i=1}^{4}e_2(a_i)+\frac{\alpha(1-\alpha)^3}{4}\sum_{i=1}^{4}e_3(a_i)+(1-\alpha)^4\sum_{i=1}^{4}e_4(a_i)))$$\\$$+\frac{1}{18}(4(18\alpha-9\alpha^2+(2-4\alpha+2\alpha^2)\sum_{i=1}^{4}e_2(a_i)+6-12\alpha+6\alpha^2)$$\\$$-9(-6\alpha^3+9\alpha^2+(2\alpha-4\alpha^2+2\alpha^3)\sum_{i=1}^{4}e_2(a_i)+6\alpha-12\alpha^2+6\alpha^3+(2-6\alpha+6\alpha^2$$\\$$-2\alpha^3)\sum_{i=1}^{4}e_2(a_i))+18(\frac{3\alpha^3}{2}-\frac{9\alpha^4}{8}+\frac{\alpha^2-2\alpha^3+\alpha^4}{2}\sum_{i=1}^{4}e_2(a_i)+\frac{3(\alpha^2-2\alpha^3+\alpha^4)}{2}$$\\$$+(\alpha-3\alpha^2+3\alpha^3-\alpha^4)\sum_{i=1}^{4}e_2(a_i)+(1-\alpha)^4\sum_{T_2(A)}e_4(\textbf{y}))).$$\\
Now, $$f(\alpha)=3\alpha+32(1-\alpha)(\frac{-1}{3}+\frac{1}{9}\sum_{T_1(A)}(-4e_2+9e_3-18e_4)(\textbf{x})$$\\$$+\frac{1}{18}\sum_{T_2(A)}(4e_2-9e_3+18e_4)(\textbf{x}))$$\\$$-32(\frac{-1}{3}+\frac{1}{9}(-4(\frac{3\alpha^2}{2}+3\alpha(1-\alpha)+(1-\alpha)^2\sum_{i=1}^{4}e_2(a_i))$$\\ $$+9(\frac{\alpha^3}{4}+\frac{3\alpha^2(1-\alpha)}{4}+\frac{\alpha(1-\alpha)^2}{2}\sum_{i=1}^{4}e_2(a_i)+(1-\alpha)^3\sum_{i=1}^{4}e_3(a_i))$$\\$$-18(\frac{\alpha^4}{64}+\frac{\alpha^3(1-\alpha)}{16}+\frac{\alpha^2(1-\alpha)^2}{16}\sum_{i=1}^{4}e_2(a_i)+\frac{\alpha(1-\alpha)^3}{4}\sum_{i=1}^{4}e_3(a_i)+(1-\alpha)^4\sum_{i=1}^{4}e_4(a_i)))$$\\$$+\frac{1}{18}(4(18\alpha-9\alpha^2+(2-4\alpha+2\alpha^2)\sum_{i=1}^{4}e_2(a_i)+6-12\alpha+6\alpha^2)$$\\$$-9(-6\alpha^3+9\alpha^2+(2\alpha-4\alpha^2+2\alpha^3)\sum_{i=1}^{4}e_2(a_i)+6\alpha-12\alpha^2+6\alpha^3+(2-6\alpha+6\alpha^2$$\\$$-2\alpha^3)\sum_{i=1}^{4}e_2(a_i))+18(\frac{3\alpha^3}{2}-\frac{9\alpha^4}{8}+\frac{\alpha^2-2\alpha^3+\alpha^4}{2}\sum_{i=1}^{4}e_2(a_i)+\frac{3(\alpha^2-2\alpha^3+\alpha^4)}{2}$$\\$$+(\alpha-3\alpha^2+3\alpha^3-\alpha^4)\sum_{i=1}^{4}e_2(a_i)+(1-\alpha)^4\sum_{T_2(A)}e_4(\textbf{y})))).$$\\

After simplifying we get a $4^{th}$ degree polynomial in $\alpha$, \\$$f(\alpha)=\frac{128}{3}+\alpha[67-80\sum_ {i=1}^{4}e_2(a_i)+80\sum_{i=1}^{4}e_3(a_i)-192\sum_{i=1}^{4}e_4(a_i)+96\sum_{T_2(A)}e_4(x)]$$\\$$+\alpha^2[-120+148\sum_{i=1}^{4}e_2(a_i)-144\sum_{i=1}^{4}e_3(a_i)+384\sum_{i=1}^{4}e_4(a_i)-192\sum_{T_2(A)}e_4(x)]$$\\$$+\alpha^3[68-88\sum_{i=1}^{4}e_2(a_i)+80\sum_{i=1}^{4}e_3(a_i)-256\sum_{i=1}^{4}e_4(a_i)+128\sum_{T_2(A)}e_4(x)]$$\\$$+\alpha^4[-15+24\sum_{i=1}^{4}e_2(a_i)-16\sum_{i=1}^{4}e_3(a_i)+64\sum_{i=1}^{4}e_4(a_i)-32\sum_{T_2(A)}e_4(x)].$$\\\\
Using the MATLAB, we obtained the minimum values of the function $f(\alpha)$ for different values of $\alpha$.  We obtained  minimum values of $f(\alpha)$ in different sub intervals of [0, 1] and the values are shown in table 1. \\
\begin{table}[h!]
\centering
\caption{Minimum value}
\begin{tabular} {| c | c| c | }
\hline
    Interval      & Minimum at  $\alpha$    & Minumum value $f(\alpha)$       \\ \hline
    [0.5, 1]       & 0.5                   & 0.1250                         \\ \hline
		
		[0.4, 1]       & 0.4                   & -0.5472          \\ \hline

    [0.45, 1]      & 0.45                  & -0.2654            \\ \hline

    [0.48, 1]     & 0.48                   & -0.0452     \\ \hline
		
		[0.485, 1]     & 0.45                   & -0.0044     \\ \hline
		
		[0.486, 1]     & 0.486                   & 0.0039     \\ \hline
		
		[0, 1]         & 0.2545                  & -0.8675    \\ \hline
		
\end{tabular}

\label{table:Minimum value}
\end{table}\\
Therefore, the function $f(\alpha) \geq 0$ in $[0.5,1].$ Hence, Lih-Wang conjecture is true for $n=4$ and in particular the inequality (3) is true in [0.486, 1].
\end{proof}

\begin{thm}
	If $A$ is a matrix in $\Omega_6$ such that the polynomials \\
	$\sum_{r=2}^{5}r\frac{(5-r)!}{6^{5-r}}{{6-r}\choose{5-r}}^2\sigma_r(A-J_6)t^{r-2}$ and  $\sum_{r=2}^{6}r\frac{(6-r)!}{6^{6-r}}\sigma_r(A-J_6)t^{r-2}$ have no roots in $(0, 1)$ then per$(tJ_6+(1-t)A)\leq t$ per$J_6+(1-t)$per$A$ for all $t\in [0.7836, 1].$
\end{thm}
\begin{proof}
$t$ per$J_6+(1-t)$per$A$-per$(tJ_6+(1-t)A)$\\ 

\noindent =$\frac{5}{324}t+(1-t)$per$A-\sum_{r=0}^{6}t^r(1-t)^{6-r}\frac{r!}{6^r}\sigma_{6-r}(A)$ \\

\noindent =$\frac{5}{324}t+(1-t)$per$A-[(1-t)^6$per$A+\frac{1}{6}t(1-t)^5\sigma_5(A)+\frac{2}{6^2}t^2(1-t)^4\sigma_4(A)+\frac{6}{6^3}t^3(1-t)^3\sigma_3(A)+\frac{24}{6^4}t^4(1-t)^2\sigma_2(A)+\frac{120}{6^5}t^5(1-t)\sigma_1(A)+\frac{720}{6^6}t^6]$\\

\noindent =$\frac{5}{324}t+$per$A(1-t)[1-(1-t)^5]-\frac{1}{6}t(1-t)^5\sigma_5(A)-\frac{1}{18}t^2(1-t)^4\sigma_4(A)-\frac{1}{36}t^3(1-t)^3\sigma_3(A)-\frac{1}{54}t^4(1-t)^2\sigma_2(A)-\frac{5}{54}t^5(1-t)-\frac{5}{324}t^6$\\

%\noindent =$\frac{5}{324}(t-t^6)-\frac{5}{54}t^5(1-t)+$per$A(1-t)[1- (1-5t+\frac{5\times4}{1\times2}t^2-\frac{5\times4\times3}{1\times2\times3}t^3+5t^4-t^5)]-\frac{1}{6}t(1-t)^5\sigma_5(A)-\frac{1}{18}t^2(1-t)^4\sigma_4(A)-\frac{1}{36}t^3(1-t)^3\sigma_3(A)-\frac{1}{54}t^4(1-t)^2\sigma_2(A)$ \\

\noindent =$\frac{5}{324}t(1-t)(1+t+t^2+t^3+t^4)-\frac{5}{54}t^5(1-t)$+per$A(1-t)(5t-10t^2+10t^3-5t^4+t^5)-\frac{1}{6}t(1-t)^5\sigma_5(A)-\frac{1}{18}t^2(1-t)^4\sigma_4(A)-\frac{1}{36}t^3(1-t)^3\sigma_3(A)-\frac{1}{54}t^4(1-t)^2\sigma_2(A)$ \\

%\noindent =$\frac{5}{324}t(1-t)(1+t+t^2+t^3-5t^4)$+per$A(1-t)(5t-10t^2+10t^3-5t^4+t^5)-\frac{1}{6}t(1-t)^5\sigma_5(A)-\frac{1}{18}t^2(1-t)^4\sigma_4(A)-\frac{1}{36}t^3(1-t)^3\sigma_3(A)-\frac{1}{54}t^4(1-t)^2\sigma_2(A)$\\

\noindent =$\frac{5}{324}t(1-t)(1+t+t^2+t^3-5t^4)$+per$At(1-t)(5-10t+10t^2-5t^3+t^4)-\frac{1}{6}t(1-t)^5\sigma_5(A)-\frac{1}{18}t^2(1-t)^4\sigma_4(A)-\frac{1}{36}t^3(1-t)^3\sigma_3(A)-\frac{1}{54}t^4(1-t)^2\sigma_2(A)$ \\

\noindent =$t(1-t)F_A(t),$ where\\

\noindent $F_A(t)=\frac{5}{324}(1+t+t^2+t^3-5t^4)$+per$A(5-10t+10t^2-5t^3+t^4)-\frac{1}{6}(1-t)^4\sigma_5(A)-\frac{1}{18}t(1-t)^3\sigma_4(A)-\frac{1}{36}t^2(1-t)^2\sigma_3(A)-\frac{1}{54}t^3(1-t)\sigma_2(A)$.\\

\noindent It is enough to prove that $F_A(t)\geq0$ for all $t\in[0.7836,1]$. Since the theorem holds good for $t=1$, it is sufficient to consider $t<1.$\\

\noindent $F_A(t)=(5-10t+10t^2-5t^3+t^4)(perA-\frac{1}{36}\sigma_5(A))+\frac{1}{36}(5-10t+10t^2-5t^3+t^4)\sigma_5(A)-\frac{1}{6}(1-t)^4\sigma_5(A)-\frac{1}{18}t(1-t)^3\sigma_4(A)-\frac{1}{36}t^2(1-t)^2\sigma_3(A)-\frac{1}{54}t^3(1-t)\sigma_2(A)+\frac{5}{324}(1+t+t^2+t^3-5t^4).$\\

\noindent By Theorem 1.2, $F_A(t)\geq \frac{1}{36}(-1+14t-26t^2+19t^3-5t^4)\sigma_5(A)-\frac{1}{18}t(1-t)^3\sigma_4(A)-\frac{1}{36}t^2(1-t)^2\sigma_3(A)-\frac{1}{54}t^3(1-t)\sigma_2(A)+\frac{5}{324}(1+t+t^2+t^3-5t^4)$\\

\noindent =$\frac{1}{36}(-1+14t-26t^2+19t^3-5t^4)(\sigma_5(A)-\frac{2}{15}\sigma_4(A))+\frac{1}{270}\sigma_4(A)(-1+14t-26t^2+19t^3-5t^4)-\frac{1}{18}t(1-t)^3\sigma_4(A)-\frac{1}{36}t^2(1-t)^2\sigma_3(A)-\frac{1}{54}t^3(1-t)\sigma_2(A)+\frac{5}{324}(1+t+t^2+t^3-5t^4).$\\

\noindent By Theorem 1.2, $F_A(t)\geq \frac{1}{270}\sigma_4(A)(-1+14t-26t^2+19t^3-5t^4)-\frac{1}{18}t(1-t)^3\sigma_4(A)-\frac{1}{36}t^2(1-t)^2\sigma_3(A)-\frac{1}{54}t^3(1-t)\sigma_2(A)+\frac{5}{324}(1+t+t^2+t^3-5t^4).$\\

\noindent This implies that $F_A(t)\geq\frac{1}{270}\sigma_4(A)\{-1+14t-26t^2+19t^3-5t^4-15t(1-3t+3t^2-t^3)\}-\frac{1}{36}t^2(1-t)^2\sigma_3(A)-\frac{1}{54}t^3(1-t)\sigma_2(A)+\frac{5}{324}(1+t+t^2+t^3-5t^4)$\\

\noindent =$\frac{1}{270}\sigma_4(A)(-1-t+19t^2-26t^3+10t^4)-\frac{1}{36}t^2(1-t)^2\sigma_3(A)-\frac{1}{54}t^3(1-t)\sigma_2(A)+\frac{5}{324}(1+t+t^2+t^3-5t^4)$\\

\noindent =$\frac{1}{270}(-1-t+19t^2-26t^3+10t^4)(\sigma_4(A)-\frac{3}{8}\sigma_3(A))+\frac{3}{2160}\sigma_3(A)(-1-t+19t^2-26t^3+10t^4)-\frac{1}{36}t^2(1-t)^2\sigma_3(A)-\frac{1}{54}t^3(1-t)\sigma_2(A)+\frac{5}{324}(1+t+t^2+t^3-5t^4)$\\

\noindent Again by Theorem 1.2, $F_A(t)\geq \frac{3}{2160}\sigma_3(A)(-1-t+19t^2-26t^3+10t^4)-\frac{1}{36}t^2(1-t)^2\sigma_3(A)-\frac{1}{54}t^3(1-t)\sigma_2(A)+\frac{5}{324}(1+t+t^2+t^3-5t^4)$ \\

\noindent =$\sigma_3(A)(\frac{-3}{2160}-\frac{3}{2160}t-\frac{3}{2160}t^2+\frac{42}{2160}t^3-\frac{30}{2160}t^4)-\frac{1}{54}t^3(1-t)\sigma_2(A)+\frac{5}{324}(1+t+t^2+t^3-5t^4)$\\

From Dokovic\cite{DZD1967} and Tverberg\cite{HTO1963} we know that for $A\in\Omega_n$,  

\noindent $\sigma_2(A)=\frac{1}{2}\sum_{i,j=1}^{n}a_{ij}^2+\frac{1}{2}n(n-2)$ and\\

\noindent $\sigma_3(A)=\frac{2}{3}\sum_{i,j=1}^{n}a_{ij}^3+\frac{1}{2}(n-4)\sum_{i,j=1}^{n}a_{ij}^2+\frac{1}{6}n(n^2-6n+10).$ \\

\noindent Therefore $F_A(t)\geq\frac{1}{2160}(-3-3t-3t^2+42t^3-30t^4)(\frac{2}{3}\sum_{i,j=1}^{6}a_{ij}^3+\sum_{i,j=1}^{6}a_{ij}^2+10)-\frac{1}{54}t^3(1-t)(\frac{1}{2}\sum_{i,j=1}^{6}a_{ij}^2+12)+\frac{5}{324}(1+t+t^2+t^3-5t^4)$\\

\noindent =$\frac{1}{3240}(-3-3t-3t^2+42t^3-30t^4)\sum_{i,j=1}^{6}a_{ij}^3+\frac{1}{2160}\sum_{i,=1}^{6}a_{ij}^2(-3-3t-3t^2+22t^3-10t^4)+\frac{1}{648}(1+t+t^2-8t^3+4t^4)$ \\

\noindent =$\sum_{i=1}^{6}\{\frac{1}{3240}(-3-3t-3t^2+42t^3-30t^4)\sum_{j=1}^{6}a_{ij}^3+\frac{1}{2160}\sum_{j=1}^{6}a_{ij}^2(-3-3t-3t^2+22t^3-10t^4)\}+\frac{1}{648}(1+t+t^2-8t^3+4t^4).$\\

\noindent It is easy to see that the coefficient of $\sum_{j=1}^{6}a_{ij}^3$ is nonnegative for all $t\in [0.6530, 1]$ and $(\sum_{j=1}^{6}a_{ij}^2)^2\leq\sum_{j=1}^{6}a_{ij}^3$ for $i=1,2,...6$ (Kopotun\cite{KAK1994}). \\

\noindent Therefore, for all $t\in[0.6530, 1], \\F_A(t)\geq\sum_{i=1}^{6}[\frac{1}{3240}(-3-3t-3t^2+42t^3-30t^4)(\sum_{j=1}^{6}a_{ij}^2)^2+\frac{1}{2160}(-3-3t-3t^2+22t^3-10t^4)\sum_{j=1}^{6}a_{ij}^2]+\frac{1}{648}(1+t+t^2-8t^3+4t^4).$\\

\noindent Put $x=\sum_{j=1}^{6}a_{ij}^2$ and\\\\$f(x)=\frac{1}{3240}(-3-3t-3t^2+42t^3-30t^4)x^2+\frac{1}{2160}(-3-3t-3t^2+22t^3-10t^4)x$.\\

\noindent Now, $f'(x)=\frac{2}{3240}x(-3-3t-3t^2+42t^3-30t^4)+\frac{1}{2160}(-3-3t-3t^2+22t^3-10t^4)$\\

\noindent It it easy to see that for $t\in[0.7836,1], f'(x)\geq 0. $\\

\noindent London and Minc(Lemma1,\cite{DLO1989}) showed that $x=\sum_{j=1}^{6}a_{ij}^2\geq\frac{1}{6}.$\\ 

\noindent Hence $f(x)$ is an increasing function on $[\frac{1}{6},\infty)$.\\

\noindent That is, $f(x)\geq f(\frac{1}{6})=\frac{-1}{3888}-\frac{1}{3888}t-\frac{1}{3888}t^2+\frac{1}{486}t^3-\frac{1}{972}t^4$\\

\noindent Therefore $t\in[0.7836,1]\\F_A(t)\geq 6(\frac{-1}{3888}-\frac{1}{3888}t-\frac{1}{3888}t^2+\frac{1}{486}t^3-\frac{1}{972}t^4)+\frac{1}{648}(1+t+t^2-8t^3+4t^4)=0$.\\

\noindent Hence the theorem.
\end{proof}
The following two corollaries are consequence of the above theorem and theorems in \cite{SSK2016}.
\begin{cor}
If $A$ is a normal matrix in $\Omega_6$ whose eigenvalues lie in the sector $[\frac{-\pi}{10},\frac{\pi}{10}]$ and the polynomial $\sum_{r=2}^{6}r\frac{(6-r)!}{6^{6-r}}\sigma_r(A-J_6)t^{r-2}$ has no root in $(0, 1)$ then \\
per$(tJ_6+(1-t)A) \leq t$ per $J_6+(1-t)$per$A$ for all $t\in [0.7836, 1].$
\end{cor}
\begin{proof}
The proof follows since if $A$ is a normal matrix in $\Omega_6$ whose eigenvalues lie in the sector $[\frac{-\pi}{10}, \frac{\pi}{10}]$ then the polynomial $\sum_{r=2}^{5}r\frac{(5-r)!}{6^{5-r}}{{6-r}\choose{5-r}}^2\sigma_r(A-J_6)t^{r-2}$ has no roots in (0, 1).
\end{proof}

\begin{cor}
If $A$ is a normal matrix in $\Omega_6$ whose eigenvalues lie in the sector $[\frac{-\pi}{12}, \frac{\pi}{12}]$ then per$(tJ_6+(1-t)A)\leq t$ per$J_6+(1-t)$per$A$ for all $t\in[0.7836, 1]$.

\end{cor}
\begin{proof}
The proof follows since if $A$ is a normal matrix in $\Omega_6$ whose eigenvalues lie in the sector $[\frac{-\pi}{12}, \frac{\pi}{12}]$ then the polynomials $\sum_{r=2}^{5}r\frac{(5-r)!}{6^{5-r}}{{6-r}\choose{5-r}}^2\sigma_r(A-J_6)t^{r-2}$ and $\sum_{r=2}^{6}r\frac{(6-r)!}{6^{6-r}}\sigma_r(A-J_6)t^{r-2}$ have no roots in (0, 1).

\end{proof}
%------------------------------------------------------------
\section{Dittert's Conjecture for $n=4$}
SG Hwang \cite{SGH1987} stated that if $A$ is a $\phi$ maximizing matrix on $K_n$ then $A$ is fully indecomposable. Based on this we prove that if $A$ is a $\phi$-maximizing matrix on $K_4$ then $A$ is fully indecomposable. We denote $R_A=(r_1, ..., r_n)$ and $C_A=(c_1, ..., c_n)$ are respectively the row and column sum vectors of $A$.

\begin{thm}
If $A$ is a $\phi$-maximizing matrix on $K_4$ then $A$ is fully indecomposble.
\end{thm}
\begin{proof}
Let $A=[a_{ij}]$ be a  $\phi$-maximizing matrix on $K_4$. First, we assume that $A$ is not fully indecomposable. Without loss of generality we may assume that $a_{12}=a_{13}=a_{14}=0$. Now we intend to show that $(a_{21},a_{31},a_{41})\neq(0,0,0).$ Suppose $a_{21}=a_{31}=a_{41}=0.$ Let $B=A(1|1)$ and $a_{11}=a$. Then $A=aI_1\oplus B$,  $\phi(A)=a\phi(B)$ and $\phi_{11}(A)=\phi(B)$. By Theorem 1.3, $a=1$ and hence $B \in K_3$ and by Theorem 1.4, $B=J_3$ and hence $\phi(A)=\phi(B)=\phi(J_3)<\phi(J_4)$, which is a contradiction. Thus $(a_{21},a_{31},a_{41})\neq(0,0,0).$\\

\indent If $a_{21}>0,a_{31}>0$ and $a_{41}>0$, by Lemma $1.2$, $\phi_{21}(A)=\phi_{31}(A)=\phi_{41}(A)$, $\phi_{21}(A)=r_1r_3r_4+c_2c_3c_4$-per$A(2|1)$ and $\phi_{31}(A)=r_1r_2r_4+c_2c_3c_4$-per$A(3|1)$. Therefore $\phi_{21}(A) =\phi_{31}(A)=o$ implies $ar_4(r_3-r_2)=0$ and hence $r_2=r_3$. Then by Lemma 1.1 (the matrix $A_1$ obtained from $A$ by replacing the 2nd and 3rd rows by its average and $A_1$ is also a $\phi$-maximizing matrix on $K_4$), we assume that $A_1$ is of the form\\ 
\begin{center}$A_1 =
\begin{bmatrix}
a&0&0&0\\x&y&z&t\\x&y&z&t\\p&q&r&s
\end{bmatrix}$, \end{center} where $a x y z t p q r s>0$. Since the second, third and fourth columns have the same (0, 1) patterns, by Lemma 1.1 we assume that $A_1$ is of the form\\
\begin{center}$A_2 =
\begin{bmatrix}
a&0&0&0\\x&y&y&y\\x&y&y&y\\p&q&q&q
\end{bmatrix}$,
\end{center} 
where $a x y p q>0$. Now since the second, third and fourth rows have the same (0, 1) patterns, by Lemma 1.1 we can assume that $A_2$ is of the form\\
\begin{center}$A_3 =
\begin{bmatrix}
a&0&0&0\\x&y&y&y\\x&y&y&y\\x&y&y&y
\end{bmatrix}$,
\end{center} 
where $a x y>0$. Suppose $x\geq y$. We construct $A_4$ from $A_3$ such that $R_{A_4}=R_{A_3}, C_{A_4}=C_{A_3}$ and per$A_4=$per$A_3$, \\ 
\begin{center}$A_4=
\begin{bmatrix}
a&0&0&0\\x&y&y&y\\x+y&0&y&y\\x-y&2y&y&y
\end{bmatrix}.$
\end{center} 
The matrix $A_4$ is also a $\phi$-maximizing matrix on $K_4$ since $R_{A_4}=R_{A_3}, C_{A_4}=C_{A_3}$ and per$A_4=$per$A_3$. Now $a_{33}>0$ and $a_{42}>0$, but $\phi_{42}(A_4)-\phi_{33}(A_4)=ay^2>0$ which is a contradiction to Lemma 1.2. Hence $x<y$. We construct $A_5$ from $A_3$ such that $R_{A_5}=R_{A_3}, C_{A_5}=C_{A_3}$ and per$A_5=$per$A_3$, \\ \\
\begin{center}$A_5=
\begin{bmatrix}
a&0&0&0\\x&y&y&y\\2x&y&y&y-x\\0&y&y&y+x
\end{bmatrix}.$
\end{center} 
The matrix $A_5$ is a $\phi$-maximizing matrix on $K_4$ since $R_{A_5}=R_{A_3}, C_{A_5}=C_{A_3}$ and per$A_5=$per$A_3$. So, by Lemma 1.1, we construct the matrix $A_6$ from $A_5$ by taking average of 3rd and 4th columns of $A_5$.  \\
\begin{center}$A_6=
\begin{bmatrix}
a&0&0&0\\x&y&y&y\\2x&y&y-\frac{x}{2}&y-\frac{x}{2}\\0&y&y+\frac{x}{2}&y+\frac{x}{2}
\end{bmatrix}$
\end{center} 
is also a $\phi$-maximizing matrix on $K_4$. $R_{A_6}=R_{A_5}, C_{A_6}=C_{A_5}$ but per$A_6=a(6y^3-\frac{yx^2}{2})<6ay^3=$per$A_5$. This implies that $\phi(A_6)>\phi(A_5)=\phi(A_3),$ which is a contradiction. Hence at least one of $a_{21},a_{31},a_{41}$ is zero.\\  
Without loss of generality we may assume that $a_{21}>0$ and $a_{31}=a_{41}=0$. Hence $A$ has the form \\
\begin{center}$A=
\begin{bmatrix}
a&0&0&0\\b&B\\0\\0
\end{bmatrix}$
\end{center} with $a b>0$. We claim that $B$ is not a positive matrix. Because if $B>0$ then we may say \\
\begin{center}$A=
\begin{bmatrix}
a&0&0&0\\b&x&x&x\\0&y&y&y\\0&z&z&z
\end{bmatrix}.$
\end{center} 
Since the third and fourth rows have the same (0, 1) patterns, we construct the matrix $\tilde{A}$ from $A$ by taking the average of 3rd and 4th rows of $A$. \\
\begin{center}$\tilde{A}=
\begin{bmatrix}
a&0&0&0\\b&x&x&x\\0&y&y&y\\0&y&y&y
\end{bmatrix}$
\end{center}
 with $a b x y>0$. $\phi_{11}(\tilde{A})-\phi_{21}(\tilde{A})=3y^2(3b-3a+7x)$. From Lemma 1.2, $\phi_{11}(\tilde{A})=\phi_{21}(\tilde{A})$, this implies $3a=3b+7x$. For a sufficiently small $\epsilon>0$, consider the perturbation matrix $\tilde{A_\epsilon}$, \\
\begin{center}$\tilde{A_\epsilon}=
\begin{bmatrix}
a-\epsilon&\epsilon&0&0\\b+\epsilon&x-\epsilon&x&x\\0&y&y&y\\0&y&y&y
\end{bmatrix}$.
\end{center}
Then $R_{\tilde{A_\epsilon}}=R_{\tilde{A}}, C_{\tilde{A_\epsilon}}=C_{\tilde{A}}$ but per$\tilde{A_\epsilon}=6axy^2-\frac{32}{3}\epsilon xy^2+O(\epsilon^2)<6axy^2$=per$(\tilde{A})$ which gives us $\phi(\tilde{A_\epsilon})>\phi(\tilde{A})$ which is a contradiction. Therefore $B$ is not a positive matrix.\\ 
\indent Since $a$per$B$=per$A>0$, we have per$B>0$ and hence without loss of generality we assume that \\
\begin{center}$B=
\begin{bmatrix}
x&l&m\\0&y&v\\0&0&z
\end{bmatrix}$,
\end{center} where $x y z l m v>0$. Consider another perturbation matrix of $A$ \\

\begin{center}$\hat{A}=
\begin{bmatrix}
a&0&0&0\\b-\epsilon&x&l&m+\epsilon\\0&0&y&v\\ \epsilon&0&0&z-\epsilon 
\end{bmatrix}.$
\end{center} 
$\hat{A}$ is a matrix in $K_4$ whose permanent value is $axy(z-\epsilon)$ which is strictly less than per$A$ = $axyz$. So $\phi (\hat{A})$ is strictly greater than $\phi(A)$ which is a contradiction to our assumption that $A$ is $\phi-$ maximizing matrix. So $A$ is fully indecomposable.
\end{proof}


\begin{thebibliography}{99}
\bibitem {RAB1965}	R. A. Brualdi and M. Newman (1965),  Inequality for permanents and permanental minors,  \textit{Proc. Cambridge Philos. Soc.},  61: 741 - 746.
\bibitem{DZD1967}D.Z. Dokovic, On a conjecture by van der Waerden, Mat. Vesnik, 19 (1967), 566-569.
\bibitem{EGM1968} Eberlein, P. J., and Govind S. Mudholkar.,Some remarks on the van der Waerden conjecture, Journal of Combinatorial Theory, 5.4, 386-396, (1968).
\bibitem{THF1979}T.H. Foregger, Remarks on a conjecture of M.Marcus and H.Minc, Linear and Multilinear Algebra, 7 (1979), 123 - 126.
\bibitem{THF1988}T.H. Foregger, Permanents of convex combinations of doubly stochastic matrices, Linear and Multilinear Algebra, 23 (1988), 79 - 90.
\bibitem{FUAH2016} Fuzhen Zhang, An update on a few permanent conjectures, Special Matrices 4.1, 305-316 (2016).
\bibitem{GSC2005} Gi-Sang Cheon and I.M. Wanless, An update on Minc's survey of open problems involving permanents, Linear Algebra Appl., 403.
(2005), 314-342.
\bibitem{FHO1964}F. Holens, Two aspects of doubly stochastic matrices: Permutation matrices and the minimum
permanent function, Ph.D. Thesis, University of Manitoba, 1964.
\bibitem{SGH1991} S.G Hwang, A Note on The Monotonicity of Permanents function, Discrete Mathematics, 91 (1991), 99-104.
\bibitem{KAK1994} Kirill.A.Kopotun, On some permanental conjectures, Linear and Multilinear Algebra, 36 (1994), 205-216.
\bibitem{KAK1996}  Kirill A. Kopotun, A Note on the Convexity of the Sum of Subpermanents, Linear Algebra and its Applications, 245:157-169(1996).
\bibitem{ETH1982}  Ko-wei Lih and E.T.H.Wang, A convexity inequality on the permanent of doubly stochastic matrices, Congressus Numerantium, vol. 36,  189 -198 (1982).
\bibitem{MMA1967}  M.Marcus and H.Minc, On a conjecture of B.L.van der Waerden, Proc. Cambridge Philos. Soc., 63 (1967), 305-309.
\bibitem{HMI1978}  H.Minc, Permanents, Encyclopedia Math. Appl., 6 (1978), Addison-Wesley, Reading, Mass.
\bibitem{SSK2016}  Subramanian.P and Somasundaram.K, Some Conjectures on Permanents of Doubly Stochastic Matrices, Journal of Discrete Mathematical Sciences and Cryptography, 19(5-6), 997-1011, (2016).
\bibitem{ETH1977}  E.T.H.Wang, On a conjecture of Marcus and Minc, Linear and multilinear Algebra, 5 (1977), 145-148.
\bibitem{HTO1963} H.Tverberg, On the permanent of a bistochastic matrix, Math Scand,12, (1963), 25-35
\bibitem{DLO1989}D.London and H.Minc, On the permanent of doubly stochastic matrices with zero diagonal, Linear and Multilinear Algebra,24 (1989), 289-300
\bibitem{IMW1999}I.M. Wanless, The Holens-Dokovic conjecture on permanents fails, Linear Algebra Appl., 286 (1999), 273 - 285.
\bibitem{RS1984}Richard Sinkhorn, A problem related to the van der waerden permanent theorem, Linear and Multilinear Algebra, 16:1-4,167-173(1984)
\bibitem{SGH1986}Suk Geun Hwang, A note on a conjecture on Permanents, Linear Algebra and its Applications, 76:31-44(1986)
\bibitem{SGH1987}Suk Geun Hwang, On a conjecture of E. Dittert, Linear Algebra and its Applications, 95:161-169(1987)
\bibitem{CW2012}G.S. Cheon, I.M. Wanless, Some results towards the Dittert conjecture on permanents, Linear Algebra and its Applications, 436: 791–801 (2012)
\bibitem{CY2006}G.S.Cheon, H.W.Yoon, A note on the Dittert conjecture for permanents, International Mathematical Forum, 1, 2006, no. 39, 1943-1949
\bibitem{HMI1978}  H.Minc, Permanents, Encyclopedia Math. Appl., 6 (1978), Addison-Wesley, Reading, Mass.
\bibitem{HSK1990} Suk Geun Hwang, Mun Gu Sohn, Si Ju Kim, The Dittert's function on a set of nonnegative matrices, International Journal of Mathematics and Mathematical Sciences,  13(4), 709-716(1990)
\bibitem{C1993}G.S.Cheon, On the monotonicity of the Dittert function on classes of nonnegative matrices, Bulletin of the Korean Mathematical Society, 30(2), 265-275(1993)


\end{thebibliography}
\end{document}